\documentclass[12pt]{amsart}
\usepackage{amsmath}	
\usepackage{amssymb}
\usepackage{tikz}
\usetikzlibrary{calc}

\newtheorem{theorem}{Theorem}[section]

\newtheorem*{theoremA}{Theorem A}
\newtheorem*{theoremB}{Theorem B}

\newtheorem{lemma}[theorem]{Lemma}

\newtheorem{corollary}[theorem]{Corollary}

\theoremstyle{remark}
\newtheorem*{remark}{Remark}

\newtheorem{examples}[theorem]{Examples}

\theoremstyle{definition}

\numberwithin{equation}{section}

\newcommand{\bC}{\mathbb{C}}
\newcommand{\bN}{\mathbb{N}}

\newcommand{\bQ}{\mathbb{Q}}
\newcommand{\bR}{\mathbb{R}}
\newcommand{\bZ}{\mathbb{Z}}

\newcommand{\cC}{{\mathcal{C}}}

\newcommand{\cO}{{\mathcal{O}}}

\newcommand{\dist}{\mathrm{dist}}
\newcommand{\et}{\quad\text{and}\quad}
\newcommand{\GL}{\mathrm{GL}}

\newcommand{\proj}{\mathrm{proj}}

\newcommand{\tbigwedge}{{\textstyle{\bigwedge}}}
\newcommand{\tL}{\tilde{L}}

\newcommand{\tuL}{\tilde{\uL}}

\newcommand{\ua}{\mathbf{a}}

\newcommand{\ub}{\mathbf{b}}

\newcommand{\ue}{\mathbf{e}}

\newcommand{\uf}{\mathbf{f}}
\newcommand{\uL}{\mathbf{L}}

\newcommand{\uP}{\mathbf{P}}

\newcommand{\uu}{\mathbf{u}}
\newcommand{\uv}{\mathbf{v}}

\newcommand{\ux}{\mathbf{x}}

\newcommand{\uy}{\mathbf{y}}

\newcommand{\uz}{\mathbf{z}}

\newcommand{\vol}{\mathrm{vol}}

\newcommand{\cOv}{\mathcal{O}_\infty}

\newcommand{\Kv}{K_\infty}
\newcommand{\norm}[1]{\|#1\|}
\newcommand{\bnorm}[1]{\big\|#1\big\|}

\newcommand{\abs}[1]{|#1|}
\newcommand{\babs}[1]{\big|#1\big|}

\newcommand{\perptop}{\perp_{\mathrm{top}}}

\DeclareMathOperator{\ord}{ord}

\begin{document}

\baselineskip=15.8pt

\title[Parametric geometry of numbers]
{Parametric geometry of numbers in function fields}
\author{Damien Roy and Michel Waldschmidt}

\thanks{Work of D.~Roy was partially supported by NSERC}

\begin{abstract}
We transpose the parametric geometry of numbers, recently created
by Schmidt and Summerer, to fields of rational
functions in one variable and analyze, in that context, the
problem of simultaneous approximation to exponential
functions.
\end{abstract}

\maketitle

\hfill\hbox{\it \small In memory of Klaus Roth}

\section{Introduction}
\label{sec:intro}

Parametric geometry of numbers is a new theory, recently created
by Schmidt and Summerer \cite{SS2009,SS2013},
which unifies and simplifies many aspects of classical Diophantine
approximation, providing a handle on problems which previously
seemed out of reach (see also \cite{Sc1982}).
Our goal is to transpose this theory to fields of rational
functions in one variable and to analyze in that context the
problem of simultaneous approximation to exponential
functions.

Expressed in the setting of \cite{R2015}, the theory deals
with a general family of convex bodies of the form
\[
 \cC_\uu({\mathrm{e}}^q)
  =\{\ux\in\bR^n\,;\,
     \|\ux\|\le 1\ \text{and}\ |\uu\cdot\ux|\le {\mathrm{e}}^{-q}\}
 \quad (q\ge 0),
\]
where the norm is the Euclidean norm, $\uu$ is a fixed unit
vector in $\bR^n$, and $\uu\cdot\ux$ denotes the scalar product
of $\uu$ and $\ux$.  For each $i=1,\dots,n$, let $L_{\uu,i}(q)$
be the logarithm of the $i$-th minimum of $\cC_\uu({\mathrm{e}}^q)$ with
respect to $\bZ^n$, that is the minimum of all $t\in\bR$ such that
${\mathrm{e}}^t\cC_\uu({\mathrm{e}}^q)$ contains at least $i$ linearly independent
elements of $\bZ^n$.  Equivalently, this is the smallest $t$
for which the solutions $\ux$ in $\bZ^n$ of
\begin{equation}
 \label{intro:eq:Lui}
 \|\ux\|\le {\mathrm{e}}^t \et |\uu\cdot\ux|\le {\mathrm{e}}^{t-q}
\end{equation}
span a subspace of $\bQ^n$ of dimension at least $i$.
Define
\begin{equation}
 \label{intro:eq:mapLu}
 \begin{array}{rl}
  \uL_\uu\colon [0,\infty) &\longrightarrow \bR^n\\
  q &\longmapsto (L_{\uu,1}(q),\dots,L_{\uu,n}(q)).
 \end{array}
\end{equation}
Although the behavior of the maps $\uL_\uu$ may be complicated
(even for $n=2$, see \cite{Ke2016}), it happens that, modulo the
additive group of bounded functions from $[0,\infty)$ to
$\bR^n$, their classes are the same as those of simpler
functions called $n$-systems, defined as follows.

An $n$-system on $[0,\infty)$ is a map
$\uP=(P_1,\dots,P_n)\colon [0,\infty) \to \bR^n$ with the
property that, for each $q\ge 0$,
\begin{itemize}
\item[(S1)] we have $0\le P_1(q)\le\cdots\le P_n(q)$ and
  $P_1(q)+\cdots+P_n(q)=q$,
\smallskip
\item[(S2)] there exist $\epsilon>0$ and integers
  $k,\ell\in\{1,\dots,n\}$ such that
  \[
   \uP(t)=\begin{cases}
         \uP(q)+(t-q)\ue_\ell
          &\text{when \ $\max\{0,q-\epsilon\}\le t\le q$,}\\
         \uP(q)+(t-q)\ue_k
          &\text{when \ $q\le t\le q+\epsilon$,}
         \end{cases}
   \]
   where \ $\ue_1=(1,0,\dots,0),\,\dots\,,\,\ue_n=(0,\dots,0,1)$,
\item[(S3)] if $q>0$ and if the
  integers $k$ and $\ell$ from (S2) satisfy $k>\ell$,
  then $P_\ell(q)=\cdots=P_k(q)$.
\end{itemize}
By \cite[Theorems 8.1 and 8.2]{R2015}, there is
an explicit constant $C(n)$, depending only on $n$,
such that, for each unit vector $\uu\in\bR^n$, there
exists an $n$-system $\uP$ on $[0,\infty)$ such that
$\norm{\uL_\uu(q)-\uP(q)}\le C(n)$ for each $q\ge 0$,
and conversely, for each $n$-system $\uP$ on $[0,\infty)$,
there exists a unit vector $\uu\in\bR^n$ with the same
property.

Instead of $\bZ$, we work here with a ring of polynomials
$A=F[T]$ in one variable $T$ over an arbitrary field $F$.
We denote by $K=F(T)$ its field of quotients equipped with
the absolute value given by
\[
 \abs{f/g} = \exp(\deg(f)-\deg(g))
\]
for any $f,g\in A$ with $g\neq 0$ (using the convention
that $\deg(0)=-\infty$ and $\exp(-\infty)=0$). The role
of $\bR$ is now played by the completion $\Kv=F((1/T))$
of $K$ with respect to that absolute value.  The extension
of this absolute value to $\Kv$ is also denoted $\abs{\ }$.
We fix an integer $n\ge 2$ and still denote by $(\ue_1,\dots,\ue_n)$
the canonical basis of $\Kv^n$. We endow $\Kv^n$ with
the maximum norm
\[
 \norm{\ux}
 =\max\{\abs{x_1},\dots,\abs{x_n}\}
  \quad\text{if}\quad
  \ux=(x_1,\dots,x_n).
\]
We also use the non-degenerate bilinear form on $\Kv^n\times\Kv^n$
mapping a pair $(\ux,\uy)$ to
\begin{equation}
 \label{intro:eq:pairing}
 \ux\cdot\uy=x_1y_1+\cdots+x_ny_n
 \quad\text{if}\quad
  \ux=(x_1,\dots,x_n)
  \et
  \uy=(y_1,\dots,y_n).
\end{equation}
This identifies $\Kv^n$ with its dual \emph{isometrically}
in the sense that
\[
 \norm{\ux}=\max\{\abs{\ux\cdot\uy}\,;\, \uy\in\Kv^n
  \ \text{and}\ \norm{\uy}\le 1\}
\]
for any $\ux\in\Kv^n$.
For a given $\uu\in \Kv^n$ of norm $1$,
for each $i=1,\dots,n$ and each $q\ge 0$, we define
$L_{\uu,i}(q)$ to be the minimum of all $t\ge 0$ for which
the solutions $\ux$ in $A^n$ of the inequalities
\eqref{intro:eq:Lui}, interpreted in $\Kv^n$,
span a subspace of $K^n$ of dimension at least $i$.
This minimum exists as we may restrict to values of $t$
in $\bZ$ or in $q+\bZ$.
Then we form a map $\uL_\uu\colon[0,\infty)\to\bR^n$
as in \eqref{intro:eq:mapLu} above.  Our first main
result reads as follows.

\begin{theoremA}
The set of maps $\uL_\uu$ where $\uu$ runs through the
elements of $\Kv^n$ of norm $1$ is the same as
the set of $n$-systems $\uP$ on $[0,\infty)$ with
$\uP(q)\in\bZ^n$ for each integer $q\ge 0$.
\end{theoremA}

As we will see in the next section, when $q$ belongs
to the set $\bN=\{0,1,2,\dots\}$ of non-negative
integers, the numbers $L_{\uu,1}(q),\dots,L_{\uu,n}(q)$
are the logarithms of the successive minima of a
convex body $\cC_\uu({\mathrm{e}}^q)$ of $\Kv^n$ with
respect to $A^n$, as defined by Mahler in
\cite{Ma1941}.  However, in terms of the inequalities
\eqref{intro:eq:Lui}, these functions naturally
extend to all real numbers $q\ge 0$.

The proof of Theorem A is similar to that of the
previously mentioned result over $\bQ$, but much simpler
in good part because, as Mahler proved in the same paper
\cite{Ma1941}, the analog of Minkowski's second convex body
theorem holds with an equality in that setting.
There is also the fact that the group of
isometries of $\Kv^n$ is an open set in $\GL_n(\Kv)$
thus in that sense much larger than the orthogonal
group of $\bR^n$. In Sections \ref{sec:constraints} and
\ref{sec:inv}, we give a
complete proof of Theorem A following \cite{R2015}.
The fact that each map $\uL_\uu$ is an $n$-system is
an adaptation of the argument of Schmidt and Summerer
in \cite[Section 2]{SS2013}.  In Section \ref{sec:duality},
we also connect the maps $\uL_\uu$ with the analogue
of those considered by these authors in \cite{SS2013}.

Because of the condition (S1), an $n$-system
$\uP=(P_1,\dots,P_n)$ on $[0,\infty)$ mapping
$\bN$ to $\bN^n$ satisfies
\[
 P_1(q)
  \le \left\lfloor\frac{q}{n}\right\rfloor
  \le \left\lceil\frac{q}{n}\right\rceil
  \le P_n(q)
 \quad
 \text{for each $q\in\bN$.}
\]
It happens that there is exactly one such $n$-system
for which
\begin{equation}
 \label{intro:eq:extremalP}
 P_1(q)=\left\lfloor\frac{q}{n}\right\rfloor
  \et
 P_n(q)=\left\lceil\frac{q}{n}\right\rceil
  \quad
 \text{for each $q\in\bN$.}
\end{equation}
When $q\equiv 0 \mod n$, such a system necessarily has
$P_1(q)=\dots=P_n(q)=q/n$.  Figure \ref{intro:fig}
shows the union of the graphs of $P_1,\dots,P_n$ over
an interval of the form $[mn,(m+1)n]$ with $m\in\bN$.
\begin{figure}[h]
     \begin{tikzpicture}[scale=1]
       \draw[-] (2,0.4)--(2,1);
       \draw[dotted] (2,1)--(2,1.6);
       \draw[->] (2,1.6)--(2,5);
       \draw[line width=0.04cm] (2.1,2)--(1.9,2) node[left]{$m$};
       \draw[dashed] (2.1,2)--(4,2);
       \draw[thick] (4,2)--(9.6,2);
       \draw[dotted,thick] (9.6,2)--(10.4,2);
       \draw[thick] (10.4,2)--(12,2);
       \draw[line width=0.04cm] (2.1,4)--(1.9,4) node[left]{$m+1$};
       \draw[dashed] (2.1,4)--(6,4);
       \draw[thick] (6,4)--(9.6,4);
       \draw[dotted,thick] (9.6,4)--(10.4,4);
       \draw[thick] (10.4,4)--(14,4);
       \draw[-] (1.9,0.5)--(2.5,0.5);
       \draw[dotted] (2.5,0.5)--(3.1,0.5);
       \draw[-] (3.1,0.5)--(9.6,0.5);
       \draw[dotted] (9.6,0.5)--(10.4,0.5);
       \draw[->] (10.4,0.5)--(15.5,0.5) node[below]{$q$};
       \draw[line width=0.04cm] (4,0.6)--(4,0.4) node[below]{$\phantom{1}mn\phantom{1}$};
       \draw[dashed] (4,0.6)--(4,2);
       \draw[line width=0.04cm] (6,0.6)--(6,0.4) node[below]{$mn+1$};
       \draw[dashed] (6,0.6)--(6,4);
       \draw[line width=0.04cm] (8,0.6)--(8,0.4) node[below]{$mn+2$};
       \draw[dashed] (8,0.6)--(8,4);
       \draw[line width=0.04cm] (12,0.6)--(12,0.4) node[below]{$mn+n-1\ $};
       \draw[dashed] (12,0.6)--(12,4);
       \draw[line width=0.04cm] (14,0.6)--(14,0.4) node[below]{$\phantom{1}mn+n$};
       \draw[dashed] (14,0.6)--(14,4);
       \foreach \q in {4,6,8,12}
         \node[circle, inner sep=0.75pt, fill] at (\q,2) {};
       \foreach \q in {6,8,12,14}
         \node[circle, inner sep=0.75pt, fill] at (\q,4) {};
       \foreach \q in {4,6,12}
         \draw[thick] (\q,2)--(\q+2,4);
       \draw[thick] (8,2)--(9,3);
       \draw[thick, dotted] (9,3)--(9.3,3.3);
       \draw[thick] (11,3)--(12,4);
       \draw[thick, dotted] (10.7,2.7)--(11,3);
       \node[above] at (4.5,2.8) {$P_n$};
       \node[above] at (6.5,2.8) {$P_{n-1}$};
       \node[above] at (8.5,2.8) {$P_{n-2}$};
       \node[above] at (10.5,2.8) {$P_{2}$};
       \node[above] at (12.5,2.8) {$P_1$};
     \end{tikzpicture}
\caption{The combined graph of the $n$-system satisfying ($\ref{intro:eq:extremalP}$).}
\label{intro:fig}
\end{figure}
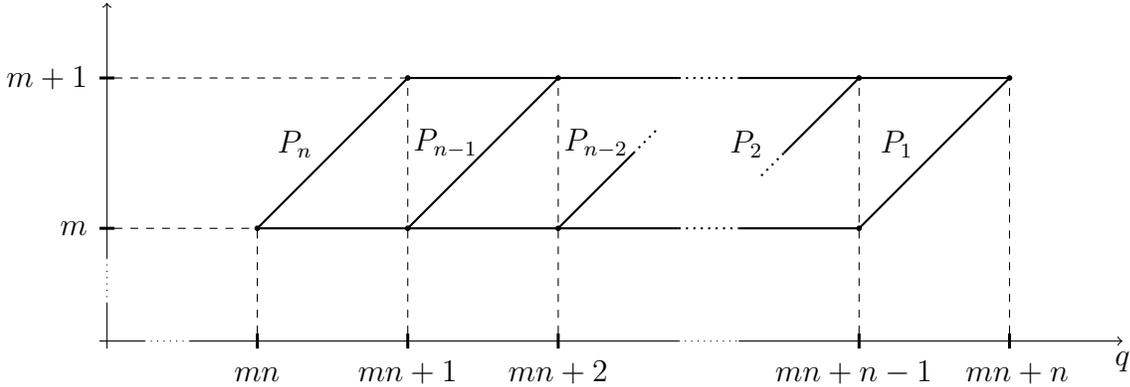
Over such an interval, the $i$-th component $P_i$
of $\uP$ is constant equal to $m$ on $[mn,mn+n-i]$, then
increases with slope $1$ on $[mn+n-i,mn+n-i+1]$ and finally
is constant equal to $m+1$ on $[mn+n-i+1,mn+n]$.

One can also characterize that system as the unique one for
which $P_n(q)-P_1(q)\le 1$ for each $q\ge 0$.  Our second main
result is the following.

\begin{theoremB}
Suppose that $F$ has characteristic zero. Let
$\omega_1,\dots,\omega_n$ be distinct elements of $F$,
and let
\[
 \uu
 = \left(
   {\mathrm{e}}^{\omega_1/T}, \dots, {\mathrm{e}}^{\omega_n/T}
   \right)
\quad
\text{where}
\quad
 {\mathrm{e}}^{\omega/T}
  =\sum_{j=0}^\infty \frac{\omega^j}{j!}T^{-j}
  \in F[[1/T]]
 \quad
 (\omega\in F).
\]
Then, we have $\norm{\uu}=1$ and the $n$-system
$\uP=\uL_\uu$ is characterized by the property
\eqref{intro:eq:extremalP}.
\end{theoremB}

As we will show in section \ref{sec:perfect}, this result in fact
extends to all perfect systems of series in the sense
of Mahler-Jager \cite{Ma1968,Ja1964}.

In 1964, A.~Baker showed that, in the notation of Theorem B,
the $n$-tuple
$\big( {\mathrm{e}}^{\omega_1/T}, \dots, {\mathrm{e}}^{\omega_n/T} \big)$
provides a counterexample to the analogue in $\bC((1/T))$ of
a conjecture of Littlewood.  In Section \ref{sec:adelic},
we generalize this result to several places of $\bC(T)$.

\section{Constraints on the successive minima}
\label{sec:constraints}

In this section, we prove that the maps $\uL_\uu$ which
appear in Theorem A are $n$-systems.  The argument
is based on the ideas of Schmidt and Summerer
in \cite{SS2013}, but follows the presentation in
\cite[\S 2]{R2015}.

\subsection{Convex bodies}

We fix an integer $n\ge 1$ and denote by
\[
 \cOv=\{ x\in\Kv\,;\, \abs{x}\le 1\}=F[[1/T]]
\]
the ring of integers of $\Kv$.  A \emph{convex body}
of $\Kv^n$ is simply a free sub-$\cOv$-module of $\Kv^n$
of rank $n$. This seemingly narrow notion, the
analog of a parallelotope, is explained by Mahler
in \cite{Ma1941}.  For example, the unit ball $\cOv^n$
of $\Kv^n$ for the maximum norm is a convex body.

Let $\cC$ be an arbitrary convex body of $\Kv^n$.
Its \emph{volume} $\vol(\cC)$ is defined
as the common value $\abs{\det(\psi)}$ attached
to all $\Kv$-linear automorphisms $\psi$ of $\Kv^n$
for which $\psi(\cOv^n)=\cC$.
For each $i=1,\dots,n$, the $i$-th minimum of $\cC$ (with respect
to $A^n$) is defined as the smallest number $\abs{\rho}$
where $\rho$ runs through the elements of $\Kv^\times$ for which
the dilated convex body
\[
 \rho\cC=\{\rho\ux\,;\,\ux\in\cC\}
\]
contains at least $i$ linearly independent elements
of $A^n$.  Since $\rho\cC$ depends only on the class
$\rho\cOv^\times$ in $\Kv^\times/\cOv^\times$, we may restrict to
elements of the form $\rho=T^a$ with $a\in\bZ$.  In this context,
Mahler's extension of Minkowski's convex body theorem
in \cite[\S 9]{Ma1941}, reads as follows (compare with
the version proved by J.~Thunder over an arbitrary function
field in \cite{Th1995}).

\begin{theorem}
\label{thm:Mahler:Minkowski}
For $i=1,\dots,n$, let $\lambda_i={\mathrm{e}}^{\mu_i}$ be the $i$-th minimum
of $\cC$.  Then we have
\[
 \lambda_1\cdots\lambda_n\vol(\cC)=1.
\]
Moreover, there exists a basis $(\ux_1,\dots,\ux_n)$ of $A^n$ over $A$
such that $\ux_i\in T^{\mu_i}\cC$ for $i=1,\dots,n$.
\end{theorem}

The last property is expressed by saying that $\ux_1,\dots,\ux_n$
\emph{realize} the successive minima
$\lambda_1\le\dots\le\lambda_n$ of\/ $\cC$.

Mahler defines the \emph{dual} or \emph{polar body} to $\cC$ by
\[
 \cC^*=\{ \uy\in \Kv^n \,;\, \abs{\ux\cdot\uy}\le 1 \text{ for all } \ux\in\cC \}.
\]
This is a convex body of $\Kv^n$ with $\vol(\cC^*)=\vol(\cC)^{-1}$.
On the algebraic counterpart, for any basis
$(\ux_1,\dots,\ux_n)$ of $A^n$, there is a dual basis
$(\ux^*_1,\dots,\ux^*_n)$ of $A^n$ characterized by
$\ux^*_i\cdot\ux_j=\delta_{i,j}$ $(1\le i\le j\le n)$.
In \cite[\S 10]{Ma1941}, Mahler shows the following.

\begin{theorem}
\label{thm:Mahler:duality}
In the notation of the previous theorem,
the successive minima of\/ $\cC^*$ are
$\lambda_n^{-1}\le\cdots\le\lambda_1^{-1}$,
realized by the elements of the dual basis
to $(\ux_1,\dots,\ux_n)$ listed in reverse
order $\ux^*_n,\dots,\ux^*_1$.
\end{theorem}

Mahler's original theory of compound bodies (over $\bR$)
also extends to the present
setting.  To state the result, fix $m\in\{1,\dots,n\}$ and put
$N=\binom{n}{m}$.  We identify $\bigwedge^m\Kv^n$ with $\Kv^N$
via a linear map sending the $N$ products
$\ue_{i_1}\wedge\cdots\wedge\ue_{i_m}$ with
$1\le i_1<\cdots<i_m\le n$ to the elements of the canonical
basis of $\Kv^N$ in some order.  Then, the
sub-$A$-module $\bigwedge^m A^n$ of $\bigwedge^m\Kv^n$
generated by the products $\uv_1\wedge\cdots\wedge\uv_m$ with
$\uv_1,\dots,\uv_m\in A^n$ is identified with $A^N$.
The \emph{$m$-th compound body} of $\cC$, denoted
$\bigwedge^m\cC$, is the sub-$\cOv$-module
of $\bigwedge^m\Kv^n$ spanned by the products
$\uv_1\wedge\cdots\wedge\uv_m$ with
$\uv_1,\dots,\uv_m\in\cC$.  This is a convex
body in that space and an adaptation of the argument
of Mahler in \cite{Ma1955} yields the following.

\begin{theorem}
\label{thm:Mahler:compound}
In the notation of the previous theorems,
the successive minima of $\bigwedge^m\cC$ are
the $N$ products $\lambda_{i_1}\cdots\lambda_{i_m}$
with $1\le i_1<\cdots<i_m\le n$, listed in monotone
increasing order.  They are realized by the products
$\ux_{i_1}\wedge\cdots\wedge\ux_{i_m}$ listed in
the corresponding order.
\end{theorem}

In particular, if $1\le m<n$, the first two minima of
$\bigwedge^m\cC$ are $\lambda_1\cdots\lambda_m$ and
$\lambda_1\cdots\widehat{\lambda_m}\lambda_{m+1}$.

 \subsection{Isometries and orthogonality}
\label{ssec:isometries}

Let $n\ge 1$ be an integer. An \emph{isometry} of $\Kv^n$
is a norm-preserving $\Kv$-linear map from $\Kv^n$ to itself.
We say that subspaces $V_1,\dots,V_\ell$ of $\Kv^n$ are
\emph{(topologically) orthogonal} if
\[
 \norm{\uv_1+\dots+\uv_\ell}
   =\max\{\norm{\uv_1},\dots,\norm{\uv_\ell}\}
\]
for any choice of $\uv_i\in V_i$ for $i=1,\dots,\ell$.
We write
\[
 \Kv^n=V_1\perptop\cdots\perptop V_\ell
\]
when $\Kv^n$ is the direct sum of such subspaces.  We say that a
finite sequence $(\uv_1,\dots,\uv_\ell)$ of elements of $V$
is \emph{orthogonal} if the one-dimensional subspaces
$\Kv\uv_1,\dots,\Kv\uv_\ell$ that they span
are orthogonal.  We say that it is \emph{orthonormal} if
moreover $\norm{\uv_i}=1$ for each $i=1,\dots,\ell$.  Thus a basis
$(\uv_1,\dots,\uv_n)$ of $\Kv^n$ over $\Kv$ is orthonormal if and
only if it is a basis of $\cOv^n$ as an $\cOv$-module.
Since $\cOv$ is a principal ideal domain,
any orthonormal sequence in $\Kv^n$
can be extended to an orthonormal basis of $\Kv^n$.

We recall that Hadamard's inequality extends naturally
to the present setting and provides a criterion for
orthogonality.

\begin{lemma}
\label{lemma:Hadamard}
Let $\ux_1,\dots,\ux_m$ be non-zero elements of $\Kv^n$.
Then, we have
\begin{equation}
\label{eq:Hadamard}
 \norm{\ux_1\wedge\cdots\wedge\ux_m}
 \le \norm{\ux_1}\cdots\norm{\ux_m}
\end{equation}
with equality if and only if $(\ux_1,\dots,\ux_m)$
is orthogonal.
\end{lemma}

\subsection{The map $\uL_\uu$}
\label{ssec:Lu}
Suppose $n\ge 2$, and let $\uu\in\Kv^n$ with
$\norm{\uu}=1$.  We now adapt the arguments of Schmidt and
Summerer in \cite[\S 2]{SS2013} to show that the
corresponding map $\uL_\uu\colon[0,\infty)\to\bR^n$
defined in the introduction is an $n$-system.

We first choose an orthonormal basis
$(\uu_1,\dots,\uu_n)$ of $\Kv^n$ ending
with $\uu_n=\uu$. Since the dual basis
$(\uu^*_1,\dots,\uu^*_n)$ is orthonormal, we obtain
an orthogonal sum decomposition
\[
 \Kv^n= U \perptop W
 \quad\text{where}\quad
 U=\langle\uu^*_1,\dots,\uu^*_{n-1}\rangle_{\Kv}
 \et
 W=\langle\uu^*_n\rangle_{\Kv}.
\]
Let $\proj_W$ denote the projection onto $W$.  For
each integer $q\ge 0$, we define
\begin{align*}
 \cC_\uu({\mathrm{e}}^q)
  &=\cOv\uu^*_1\oplus\cdots\oplus\cOv\uu^*_{n-1}
      \oplus\cOv T^{-q}\uu^*_n \\
  &=\{\ux\in\Kv^n\,;\,
      \norm{\ux}\le 1\ \text{and}\
      \norm{\proj_W(\ux)}\le {\mathrm{e}}^{-q}\} \\
  &=\{\ux\in\Kv^n\,;\,
      \norm{\ux}\le 1\ \text{and}\
      \abs{\uu\cdot\ux}\le {\mathrm{e}}^{-q}\}.
\end{align*}
The first equality shows that this is a convex body
of $\Kv^n$ of volume ${\mathrm{e}}^{-q}$.  The last one implies
that, for each $j=1,\dots,n$, its $j$-th minimum is
$\exp(L_{\uu,j}(q))$ where $L_{\uu,j}(q)$ is
defined in the introduction.

Now, fix an integer $m$ with $1\le m< n$.  Put
$N=\binom{n}{m}$ and $M=\binom{n-1}{m-1}$.  We denote
by $\omega_1,\dots,\omega_{N-M}$ the products
$\uu^*_{i_1}\wedge\cdots\wedge\uu^*_{i_m}$
with $1\le i_1<\cdots<i_m<n$ in some order
and by $\omega_{N-M+1},\dots,\omega_N$ those
with $1\le i_1<\cdots<i_m=n$.  Since
$(\omega_1,\dots,\omega_N)$ is an orthonormal
basis of $\bigwedge^m\Kv^n$, we deduce that
\begin{align*}
 \tbigwedge^m\cC_{\uu}({\mathrm{e}}^q)
  &=\left( \cOv\omega_1\oplus\cdots\oplus\cOv\omega_{N-M} \right)
    \oplus
    \left( \cOv T^{-q}\omega_{N-M+1}
         \oplus\cdots\oplus\cOv T^{-q}\omega_N \right)\\
  &=\big\{\omega\in\tbigwedge^m\Kv^n \,;\,
      \norm{\omega}\le 1\ \text{and}\
      \norm{\proj_{W^{(m)}}(\omega)}\le {\mathrm{e}}^{-q}\big\},
\end{align*}
where the projection is taken with respect to the
decomposition
\[
 \tbigwedge^m\Kv^n = U^{(m)}\perptop W^{(m)}
 \quad\text{with}\quad
 U^{(m)}=\tbigwedge^m U
 \et
 W^{(m)}=\big(\tbigwedge^{m-1} U\big) \wedge W.
\]
In particular, $\tbigwedge^m\cC_{\uu}({\mathrm{e}}^q)$ has volume ${\mathrm{e}}^{-Mq}$.
For each $j=1,\dots,N$ and each $q\ge 0$, we define
$L^{(m)}_{\uu,j}(q)$ to be the minimum of all
$t\ge 0$ for which the inequalities
\begin{equation}
 \label{Lu:eq:Lm1}
 \norm{\omega}\le {\mathrm{e}}^t
 \et
 \norm{\proj_{W^{(m)}}(\omega)}\le {\mathrm{e}}^{t-q}
\end{equation}
admit at least $j$ linearly independent solutions
$\ux$ in $\bigwedge^mA^n$.  When $q\in\bN$, this is
the logarithm of the $j$-th minimum of $\bigwedge^m\cC_{\uu}({\mathrm{e}}^q)$.
In general, the minimum exists because we may restrict
to values of $t$ in $\bZ\cup(q+\bZ)$.
In the case where $m=1$, we have $N=n$ and
$L_{\uu,j}^{(1)}=L_{\uu,j}$ for $j=1,\dots,n$.

Note that, for fixed $q\ge 0$, the points
$\omega_1,\dots,\omega_N$ satisfy \eqref{Lu:eq:Lm1}
for the choice of $t=q$, thus
\begin{equation}
 \label{Lu:eq:monotone}
 0 \le L^{(m)}_{\uu,1}(q)\le \cdots
   \le L^{(m)}_{\uu,N}(q)\le q
 \quad
 (q\ge 0).
\end{equation}
We also note that, for each $j=1,\dots,N$, we have
\[
 L_{\uu,j}^{(m)}(q_1)
   \le L_{\uu,j}^{(m)}(q_2)
   \le (q_2-q_1)+L_{\uu,j}^{(m)}(q_1)
 \quad
 \text{when}
 \quad
 0\le q_1\le q_2.
\]
Thus, $L_{\uu,1}^{(m)},\dots,L_{\uu,N}^{(m)}$ are
continuous functions on $[0,\infty)$.  We make
additional observations.

\begin{lemma}
 \label{lemma:Lm}
For each $a>0$, the union of the graphs of
$L^{(m)}_{\uu,1},\dots,L^{(m)}_{\uu,N}$ over $[0,a]$ is
contained in the union of the graphs of finitely
many functions
\[
 \begin{array}{rl}
  L_\omega\colon [0,\infty)&\longrightarrow \bR\\
  q&\longmapsto
  L_\omega(q)= \max\{\,\log\norm{\omega},\,
     q+\log\norm{\proj_{W^{(m)}}(\omega)}\,\}
 \end{array}
\]
associated to non-zero points $\omega$ in $\bigwedge^mA^n$.
\end{lemma}

For $\omega \in \bigwedge^mA^n\setminus\{0\}$ and $q\ge 0$,
the number $L_\omega(q)$ is the smallest real number
$t\ge 0$ satisfying \eqref{Lu:eq:Lm1}.  In particular,
when $q\in\bN$, it is the smallest integer $t$ such that
$\omega\in T^t\bigwedge^m\cC_\uu({\mathrm{e}}^q)$.  As
this measures the distance from $\omega$ to
$\bigwedge^m\cC_\uu({\mathrm{e}}^q)$ for varying $q$, we say that the graph of
$L_\omega$ is the \emph{trajectory} of $\omega$.
In the case $m=1$, the trajectory of a non-zero point
$\ux$ in $\bigwedge^1 A^n = A^n$ is the graph of the map
\begin{equation}
 \label{eq:trajectoire_x}
 \begin{array}{rl}
  L_\ux\colon [0,\infty)&\longrightarrow \bR\\
  q&\longmapsto
  L_\ux(q)= \max\{\,\log\norm{\ux},\, q+\log\abs{\uu\cdot\ux}\,\}\,.
 \end{array}
\end{equation}

\begin{proof}[Proof of Lemma \ref{lemma:Lm}]
Fix a choice of $a>0$.  By \eqref{Lu:eq:monotone},
the union of the graphs of $L^{(m)}_{\uu,1}, \dots,
L^{(m)}_{\uu,N}$ over $[0,a]$ is contained
in $[0,a]\times [0,a]$.
By construction, it is also contained in the union
of the trajectories of the non-zero points $\omega$
in $\bigwedge^mA^n$.  The conclusion follows because,
for such $\omega$, we have
$\log\norm{\omega}\in \bN$ and
$\log\norm{\proj_{W^{(m)}}(\omega)}\in\bZ\cup\{-\infty\}$.
Thus, there are only finitely many possible trajectories
meeting $[0,a]\times[0,a]$.
\end{proof}

\begin{lemma}
 \label{lemma:Lm:slopes}
For $j=1,\dots,N$, the map $L^{(m)}_{\uu,j}$
is continuous and piecewise linear with
constant slope $0$ or $1$ on each interval
of the form $[a,a+1]$ with $a\in\bN$.
Moreover, for each $q\ge 0$, we have
\begin{itemize}
 \item[(i)]
   $L^{(m)}_{\uu,1}(q)+\cdots+L^{(m)}_{\uu,N}(q)=Mq$,
 \item[(ii)]
   $L^{(m)}_{\uu,1}(q)
  = L_{\uu,1}(q)+\cdots+L_{\uu,m}(q)$,
 \item[(iii)]
   $L^{(m)}_{\uu,2}(q)-L^{(m)}_{\uu,1}(q)
  = L_{\uu,m+1}(q)-L_{\uu,m}(q)$.
\end{itemize}
\end{lemma}

\begin{proof}
The first assertion is a direct consequence of the
previous lemma because the maps $L_\omega$ with
$\omega\in\bigwedge^mA^n\setminus\{0\}$ are piecewise
linear with constant slope $0$ or $1$ in the
intervals between consecutive integers, and
we already know that the maps $L^{(m)}_{\uu,j}$
are continuous.

When $q$ is an integer, the equality (i) follows
from Theorem \ref{thm:Mahler:Minkowski} applied
to the convex body $\bigwedge^m\cC_{\uu}({\mathrm{e}}^q)$ of
$\bigwedge^m\Kv^n$ while (ii) and (iii) follow
from Theorem \ref{thm:Mahler:compound} together
with the remark stated below that theorem.
The three equalities then extend
to all $q\ge 0$ because all the functions
involved have a constant slope
between consecutive integers.
\end{proof}

\begin{lemma}
\label{lemma:Lm:main}
Suppose that $L^{(m)}_{\uu,1}$ changes slope
from $1$ to $0$ at some point $q>0$, then
$q$ is an integer and we have
$L_{\uu,m}(q)=L_{\uu,m+1}(q)$.
\end{lemma}

\begin{proof}
Put $a=L^{(m)}_{\uu,1}(q)$.  By the preceding lemmas,
the point $q$ is an integer and there exist
$\alpha,\beta\in \bigwedge^mA^n\setminus\{0\}$ such that
\[
 L^{(m)}_{\uu,1}(t)
  = \left\{
    \begin{aligned}
    a+t-q &= L_\alpha(t) &\text{if $q-1\le t\le q$,}\\
    a &= L_\beta(t) &\text{if $q\le t\le q+1$.}
    \end{aligned}
    \right.
\]
Since $L_\beta$ changes slope at most once on $[0,\infty)$,
going from slope $0$ to slope $1$, we deduce that $L_\beta$
is constant equal to $a$ on $[0,q+1]$.  In particular,
$L_\beta-L_\alpha$ is not constant on $[q-1,q]$.  So
$\alpha$ and $\beta$ are linearly independent, and thus
$L^{(m)}_{\uu,2}(q)=a=L^{(m)}_{\uu,1}(q)$.  The conclusion
then follows from Lemma \ref{lemma:Lm:slopes} (iii).
\end{proof}

\begin{theorem}
\label{theorem:nsys}
The map $\uL_\uu=(L_{\uu,1},\dots,L_{\uu,n}) \colon
[0,\infty)\to\bR^n$ is an $n$-system.
\end{theorem}

\begin{proof}
For the choice of $m=1$, the inequalities
\eqref{Lu:eq:monotone} and the identity of
Lemma \ref{lemma:Lm:slopes} (i) become
\[
 0 \le L_{\uu,1}(q)\le \cdots
   \le L_{\uu,n}(q)\le q
 \et
 L_{\uu,1}(q)+\cdots+L_{\uu,n}(q)=q
 \quad
 (q\ge 0).
\]
Thus $\uL_\uu$ satisfies the condition (S1) in
the definition of an $n$-system.  It also satisfies
(S2) because, by Lemma \ref{lemma:Lm:slopes},
each $L_{\uu,j}=L^{(1)}_{\uu,j}$ has constant slope
$0$ or $1$ in each interval $[q,q+1]$ with $q\in\bN$
while, by the above, their sum has slope $1$ on $[q,q+1]$.
So, for each $q\in\bN$, there is an index
$k\in\{1,\dots,n\}$ for which $L_{\uu,k}$ has slope $1$
on $[q,q+1]$ while the other maps $L_{\uu,j}$ with
$j\neq k$ are constant on that interval.  Now, suppose
that $q\ge 1$ and that $L_{\uu,\ell}$ has slope $1$
on $[q-1,q]$.  Suppose further that $\ell<k$.  Then,
for each integer $m$ with $\ell\le m<k$, the map
$L^{(m)}_{\uu,1}=L_{\uu,1}+\cdots+L_{\uu,m}$ changes
slope from $1$ to $0$ at $q$.  By Lemma \ref{lemma:Lm:main},
this implies that $L_{\uu,\ell}(q)=\cdots=L_{\uu,k}(q)$.
Thus (S3) holds as well.
\end{proof}

\section{The inverse problem}
\label{sec:inv}

Our goal here is to complete the proof of Theorem A by
providing a converse to Theorem \ref{theorem:nsys}.
To this end, we follow the argument of \cite{R2015}
taking advantage of the notable simplifications that
arise in the present non-archimedean setting.

\subsection{The projective distance}
\label{ssec:proj_dist}
We define the \emph{projective distance} between two non-zero points
$\ux$ and $\uy$ in $\Kv^n$ by
\[
 \dist(\ux,\uy):=\frac{\norm{\ux\wedge\uy}}{\norm{\ux}\,\norm{\uy}}.
\]
Lemma \ref{lemma:Hadamard} implies that $\dist(\ux,\uy)\le 1$
with equality if and only if the pair $(\ux,\uy)$ is
orthogonal.  Moreover, the projective distance is
invariant under an isometry of $\Kv^n$.  The next result
relates it to the distance associated with the norm
on $\Kv^n$.

\begin{lemma}
\label{lemma:comparison_distance_norm}
Let $\ux\in \Kv^n\setminus\{0\}$.  Then, there exists
$\uu\in \Kv^n$ with $\norm{\uu}=1$ such that
$\norm{\ux}=\abs{\uu\cdot\ux}$.
For any such $\uu$ and any $\uy\in \Kv^n\setminus\{0\}$ with
$\dist(\ux,\uy)<1$, we have
$\norm{\uy}=\abs{\uu\cdot\uy}$ and
\[
 \dist(\ux,\uy)=\bnorm{(\uu\cdot\ux)^{-1}\ux-(\uu\cdot\uy)^{-1}\uy}.
\]
\end{lemma}

\begin{proof}
Let $(\uu_1,\dots,\uu_n)$ be an orthonormal basis of $\Kv^n$
and let $(\ux_1,\dots,\ux_n)$ be the dual basis.
Since the latter is also orthonormal, we find
\[
 \norm{\ux}
   = \norm{(\uu_1\cdot\ux)\ux_1+\cdots+(\uu_n\cdot\ux)\ux_n}
   = \max\{ \abs{\uu_1\cdot\ux},\dots,\abs{\uu_n\cdot\ux} \}.
\]
Thus, there exists an index $i$ such that
$\abs{\uu_i\cdot\ux}=\norm{\ux}$.

Let $\uy\in \Kv^n\setminus\{0\}$.  We also note that
\begin{align*}
 \norm{\ux\wedge\uy}
   &= \max_{1\le j,k\le n}
     \abs{ (\uu_j\cdot\ux)(\uu_k\cdot\uy)
       -(\uu_j\cdot\uy)(\uu_k\cdot\ux) } \\
   &= \max_{1\le j\le n}
     \norm{ (\uu_j\cdot\ux)\uy-(\uu_j\cdot\uy)\ux }.
\end{align*}
If $\abs{\uu_1\cdot\ux}=\norm{\ux}$, we deduce
that, for each $j=1,\dots,n$,
\begin{align*}
 \norm{\ux} &\norm{(\uu_j\cdot\ux)\uy-(\uu_j\cdot\uy)\ux} \\
   &= \bnorm{(\uu_j\cdot\ux)\big((\uu_1\cdot\ux)\uy-(\uu_1\cdot\uy)\ux\big)
      +\big((\uu_j\cdot\ux)(\uu_1\cdot\uy)
            -(\uu_j\cdot\uy)(\uu_1\cdot\ux)\big)\ux}\\
   &\le \norm{\ux} \norm{(\uu_1\cdot\ux)\uy-(\uu_1\cdot\uy)\ux},
\end{align*}
and thus $\norm{\ux\wedge\uy}
= \norm{ (\uu_1\cdot\ux)\uy-(\uu_1\cdot\uy)\ux }$.
If moreover $\abs{\uu_1\cdot\uy}<\norm{\uy}$,
then we have $\norm{(\uu_1\cdot\ux)\uy}
=\norm{\ux}\norm{\uy} >\norm{(\uu_1\cdot\uy)\ux}$
and the previous formula then yields
$\norm{\ux\wedge\uy}=\norm{\ux}\norm{\uy}$, thus
$\dist(\ux,\uy)=1$.  We conclude that,
if $\abs{\uu_1\cdot\ux}=\norm{\ux}$
and $\dist(\ux,\uy)<1$, then
$\abs{\uu_1\cdot\uy}=\norm{\uy}$ and
\[
 \dist(\ux,\uy)
  = \frac{\norm{ (\uu_1\cdot\ux)\uy-(\uu_1\cdot\uy)\ux }}
         {\norm{\ux}\norm{\uy}}
  = \bnorm{ (\uu_1\cdot\ux)^{-1}\ux-(\uu_1\cdot\uy)^{-1}\uy }.
\]
The lemma follows because any element $\uu$
of $\Kv^n$ of norm $1$ can be taken as the first
component of an orthonormal basis of $\Kv^n$.
\end{proof}

This implies in particular that the projective distance
satisfies the ultrametric form of the triangle inequality, namely
\[
 \dist(\ux,\uz)\le \max\{\dist(\ux,\uy),\, \dist(\uy,\uz)\}.
\]
for any non-zero elements $\ux$, $\uy$, $\uz$ of $\Kv^n$.
This is clear if $\dist(\ux,\uy)=1$ or $\dist(\uy,\uz)=1$.
Otherwise, both numbers are $<1$ and the inequality follows
from the lemma applied to the point $\uy$.

\subsection{The key lemma}
\label{ssec:key_lemma}
The following is an adaptation of \cite[Lemma 5.1]{R2015}
which will serve to construct recursively a sequence
of bases of $A^n$ with specific properties.
Note the stronger hypothesis and conclusion.

\begin{lemma}
\label{lemma:construction_bases}
Let $h,k,\ell\in\{1,\dots,n\}$ with $h\le \ell$ and
$k<\ell$, let $(\ux_1,\dots,\ux_n)$ be a basis of $A^n$,
let $\uu\in\Kv^n$, and let $a\in \bZ$ with ${\mathrm{e}}^a > \norm{\ux_h}$
and ${\mathrm{e}}^a \ge \norm{\ux_1},\dots,\norm{\ux_\ell}$.  Suppose
that $(\ux_1, \dots, \widehat{\ux_h}, \dots, \ux_n,\uu)$
is an orthogonal basis of $\Kv^n$.
Then, there exists a basis $(\uy_1,\dots,\uy_n)$ of $A^n$
satisfying
\begin{itemize}
 \item[1)]
   $(\uy_1,\dots,\widehat{\uy_\ell\,},\dots,\uy_n)
     =(\ux_1,\dots,\widehat{\ux_h},\dots,\ux_n)$,
 \item[2)]
   $\uy_\ell \in
    \ux_h
     +\big\langle
      \ux_1,\dots,\widehat{\ux_h},\dots,\ux_\ell
      \big\rangle_A$\,,
 \item[3)]
    $\norm{\uy_\ell} = {\mathrm{e}}^a$,
 \item[4)]
    $(\uy_1, \dots, \widehat{\uy_k}, \dots,\uy_n,\uu)$
    is an orthogonal basis of $\Kv^n$,
 \item[5)]
    $\det(\uy_1, \dots, \widehat{\uy_k}, \dots,\uy_n,\uu)$
    and $\det(\ux_1, \dots, \widehat{\ux_h}, \dots, \ux_n,\uu)$
    have the same leading coefficients as elements of
    $\Kv=F((1/T))$.
\end{itemize}
\end{lemma}

Although the basis $(\uy_1,\dots,\uy_n)$ is in general
not uniquely determined by the conditions 1) to 5), the
argument that we provide below is deterministic in the
sense that, for the given data, it yields a unique basis
with the requested properties.

\begin{proof}
We use 1) as a definition of the vectors
$\uy_1,\dots,\widehat{\uy_\ell},\dots,\uy_n$.
Then, $(\uy_1,\dots,\uy_n)$ is a basis of $A^n$
for any choice of $\uy_\ell$ satisfying 2).
Since $k<\ell$, the point $\uy_k$ belongs to the set
\[
   \{\uy_1,\dots,\uy_{\ell-1}\}
    = \{\ux_1,\dots,\widehat{\ux_h},\dots,\ux_{\ell}\}
\]
and so $\norm{\uy_k}={\mathrm{e}}^{a-b}$ for some integer $b\ge 0$.
In particular the choice of
\[
 \uy_\ell = \ux_h + T^b\uy_k
\]
fulfils the condition 2).  Since $\norm{\ux_h} < {\mathrm{e}}^a =
\norm{T^b\uy_k}$, we also have $\norm{\uy_\ell}={\mathrm{e}}^a$
as requested by condition 3).  Moreover,
$(\uy_1,\dots,\widehat{\uy_\ell\,},\dots,\uy_n,\uu)$
is an orthogonal basis of $\Kv^n$.  So, we can write
\[
 \ux_h = c_\ell\uu + \sum_{j\neq \ell} c_j \uy_j
\]
with coefficients $c_1,\dots,c_n\in \Kv$ such that
$\norm{c_\ell\uu}\le \norm{\ux_h}$ and $\norm{c_j\uy_j}\le
\norm{\ux_h}$ for any $j\neq \ell$.  In particular, this yields
$\norm{c_k\uy_k} < {\mathrm{e}}^a = \norm{T^b\uy_k}$, so $\abs{c_k} <
\abs{T^b}$, and thus $\abs{T^b+c_k}={\mathrm{e}}^b$.  Since
\begin{equation}
 \label{eq:y_ell}
 \uy_\ell
  \in
 (T^{b}+c_k)\uy_k
 +\langle \uy_1,\dots,\widehat{\uy_k},\dots,
       \widehat{\uy_\ell\,},\dots,\uy_n,\uu\rangle_{\Kv},
\end{equation}
we deduce that
\[
\norm{\uy_1\wedge\cdots\wedge\widehat{\uy_k}
       \wedge\cdots\wedge\uy_n\wedge\uu}
  = {\mathrm{e}}^b\,
     \norm{\uy_1\wedge\cdots\wedge\widehat{\uy_\ell\,}
       \wedge\cdots\wedge\uy_n\wedge\uu}.
\]
As $(\uy_1,\dots,\widehat{\uy_\ell\,},\dots,\uy_n,\uu)$
is an orthogonal basis of $\Kv^n$, Lemma
\ref{lemma:Hadamard} then yields
\[
\norm{\uy_1\wedge\cdots\wedge\widehat{\uy_k}
       \wedge\cdots\wedge\uy_n\wedge\uu}
 = \frac{\norm{\uy_1}\cdots\norm{\uy_n}\,\norm{\uu}}{{\mathrm{e}}^{-b}\norm{\uy_\ell}}
 = \frac{\norm{\uy_1}\cdots\norm{\uy_n}\,\norm{\uu}}{\norm{\uy_k}}
\]
because ${\mathrm{e}}^{-b}\norm{\uy_\ell}={\mathrm{e}}^{a-b}=\norm{\uy_k}$.
By Lemma \ref{lemma:Hadamard}, this in turn implies that
the $n$-tuple
$(\uy_1,\dots,\widehat{\uy_k\,},\dots,\uy_n,\uu)$
is an orthogonal basis of $\Kv^n$.  Thus the condition 4)
is satisfied as well.  Finally, the relation \eqref{eq:y_ell}
yields
\begin{align*}
 \det(\uy_1, \dots, \widehat{\uy_k}, \dots,\uy_n,\uu)
  &= (T^b+c_k)\det(\uy_1, \dots, \widehat{\uy_\ell}, \dots,\uy_n,\uu)\\
  &= (T^b+c_k)\det(\ux_1, \dots, \widehat{\ux_h}, \dots, \ux_n,\uu).
\end{align*}
Since $T^b+c_k$ has leading coefficient $1$ in $F((1/T))$
(because $|c_k|<|T^b|$), this gives 5).
\end{proof}

We will use this lemma in combination with the following
result (cf.\ \cite[Lemma 4.7]{R2015}).

\begin{lemma}
\label{lemma:dist_dual}
Let $1\le k<\ell\le n$ be integers, let $(\uy_1,\dots,\uy_n)$
be a basis of $\Kv^n$, and let $(\uy^*_1,\dots,\uy^*_n)$
denote the dual basis of $\Kv^n$ in the sense that
$\uy^*_i\cdot\uy_j=\delta_{i,j}$ ($1\le i,j\le n$).  Assume that the
$(n-1)$-tuples $(\uy_1,\dots,\widehat{\uy_\ell\,},\dots,\uy_n)$
and $(\uy_1,\dots,\widehat{\uy_k},\dots,\uy_n)$ are
both orthogonal families in $\Kv^n$.  Then, we have
\begin{equation}
 \label{dist:lemVVH:eq2}
 \dist(\uy_k^*,\uy_\ell^*)
  = \frac{\norm{\uy_1\wedge\cdots\wedge\uy_n}}
         {\norm{\uy_1}\cdots\norm{\uy_n}}.
\end{equation}
\end{lemma}

\begin{proof}
Without loss of generality, we may assume that
$\uy_1,\dots,\uy_n$ all have norm $1$.  Upon permuting
$\uy_1$ and $\uy_k$ if $k>1$, as well as permuting
$\uy_n$ and $\uy_\ell$ if $\ell<n$, we may also
assume that $k=1$ and $\ell=n$, so that $(\uy_2,\dots,\uy_n)$
and $(\uy_1,\dots,\uy_{n-1})$ are orthonormal families.
We then need to show that $\dist(\uy_1^*,\uy_n^*)
= \norm{\uy_1\wedge\cdots\wedge\uy_n}$.

To this end, we first choose $\uu\in\Kv^n$ so that
$(\uy_1,\dots,\uy_{n-1},\uu)$ is an orthonormal basis
of $\Kv^n$.  Write $\uu=\sum_{j=1}^n c_j\uy_j$ where
$c_j=\uu\cdot\uy^*_j\in\Kv$ for $j=1,\dots,n$.  Then,
we have $c_n\neq 0$ and, applying Lemma
\ref{lemma:Hadamard} to that family, we find
\[
 1= \norm{\uy_1\wedge\cdots\wedge\uy_{n-1}\wedge\uu}
  = \abs{c_n}\,\norm{\uy_1\wedge\cdots\wedge\uy_n}.
\]
Applying the same lemma to $(\uy_2,\dots,\uy_n)$, we obtain
as well
\begin{align*}
 1&= \norm{\uy_2\wedge\cdots\wedge\uy_n} \\
  &= \bnorm{\uy_2\wedge\cdots\wedge\uy_{n-1}\wedge c_n^{-1}(\uu-c_1\uy_1)} \\
  &= \abs{c_n}^{-1}\,
    \bnorm{ \uy_2\wedge\cdots\wedge\uy_{n-1}\wedge\uu
         + (-1)^{n-1} c_1 \uy_1\wedge\cdots\wedge\uy_{n-1}} \\
  &= \abs{c_n}^{-1} \max\{1,\abs{c_1}\},
\end{align*}
where the last equality uses the fact that
$\uy_2\wedge\cdots\wedge\uy_{n-1}\wedge\uu$ and
$\uy_1\wedge\cdots\wedge\uy_{n-1}$ are orthogonal unit elements
of $\bigwedge^{n-1}\Kv^n$.  Combining these results,
we conclude that
\begin{equation}
 \label{lemma:dist_dual:eq1}
 \max\{1,\abs{c_1}\}^{-1}
  = \abs{c_n}^{-1}
  = \norm{\uy_1\wedge\cdots\wedge\uy_n}.
\end{equation}

The dual basis to $(\uy_1,\dots,\uy_{n-1},\uu)$
in $\Kv^n$ is
\[
 \left(
 \uy^*_1-\frac{c_1}{c_n}\uy^*_n,
 \dots,
 \uy^*_{n-1}-\frac{c_{n-1}}{c_n}\uy^*_n,
 \frac{1}{c_n}\uy^*_n
 \right).
\]
It is orthonormal because it is dual to an
orthonormal basis of $\Kv^n$. Then
the decompositions
\[
 \uy^*_1
  = \left( \uy^*_1-\frac{c_1}{c_n}\uy^*_n\right)
    + c_1\left(\frac{1}{c_n}\uy^*_n\right)
 \et
 \uy^*_n
  = c_n\left(\frac{1}{c_n}\uy^*_n\right),
\]
yield
\[
 \norm{\uy^*_1}=\max\{1,\abs{c_1}\},\quad
 \norm{\uy^*_n}=\abs{c_n}
 \et
 \norm{\uy^*_1\wedge\uy^*_n}=\abs{c_n},
\]
thus $\dist(\uy^*_1,\uy^*_n)=\max\{1,\abs{c_1}\}^{-1}$ and
\eqref{lemma:dist_dual:eq1} yields the conclusion.
\end{proof}

 \subsection{Construction of a point}
\label{ssec:construction}
The last lemma that we need is the following description
of the class of $n$-systems that are involved in Theorem A
(cf. \cite[\S1]{R2015}).

\begin{lemma}
\label{lemma:description_nsys}
Let $\uP=(P_1,\dots,P_n)\colon [0,\infty)\to \bR^n$ be
an $n$-system such that $\uP(q)\in\bZ^n$ for each
integer $q\ge 0$.  There exist $s\in\{1,2,\dots,\infty\}$,
and sequences of integers $(q_i)_{0\le i<s}$, $(k_i)_{0\le i<s}$
and $(\ell_i)_{0\le i<s}$, starting with $q_0=0$,
$k_0=\ell_0=n$, with the following property.  Put
$q_s=\infty$ if $s<\infty$.  Then, for each index $i$ with
$0\le i<s$, we have
\begin{itemize}
\item[(i)] $q_i<q_{i+1}$
\item[(ii)] if\/ $i>0$, then $1\le k_i<\ell_i\le n$ and
  $P_{k_i}(q_i) < P_{\ell_i}(q_i)$,
\item[(iii)] if\/ $i+1<s$, then $\ell_{i+1}\ge k_i$ and
  $P_{\ell_{i+1}}(q_{i+1}) = q_{i+1}-q_i+P_{k_i}(q_i)$,
\item[(iv)]
  $\uP(q)
     = \Phi_n\big(
       P_1(q_i),\dots,\widehat{P_{k_i}(q_i)},\dots,P_n(q_i),
       q-q_i+P_{k_i}(q_i)
       \big)$ \quad ($q_i\le q <q_{i+1}$),
\end{itemize}
where $\Phi_n\colon\bR^n \to \Delta_n:=\{(x_1,\dots,x_n)\in\bR^n\,;\,
x_1\le\cdots\le x_n\}$ is the map that lists the coordinates
of a point in monotone increasing order.
\end{lemma}

The properties (iii) and (iv) mean that the
union of the graphs of $P_1,\dots,P_n$ over the interval
$[q_i,q_{i+1})$ (called the \emph{combined graph} of $\uP$ over
that interval), consists of horizontal line segments with
ordinates $P_1(q_i),\dots,\widehat{P_{k_i}(q_i)},\dots,P_n(q_i)$
(not necessarily distinct), and a line segment of slope $1$
starting on the point $(q_i,P_{k_i}(q_i))$ and, if $i+1<s$,
ending on the point $(q_{i+1},P_{\ell_{i+1}}(q_{i+1}))$ or else
going to infinity.

\begin{proof}[Proof of Lemma \ref{lemma:description_nsys}]
By hypothesis, the function $\uP$ satisfies the conditions
(S1) to (S3) stated in the introduction.  Let $a\in\bN$.
By (S1) the sum of the coordinates of $\uP(a)\in\bN^n$ is
$a$ and the sum of those of $\uP(a+1)\in\bN^n$ is $a+1$.
Since, by (S2), each component of $\uP$ is monotone
increasing on $[0,\infty)$, we must have $\uP(a+1)=
\uP(a)+\ue_k$ for some $k\in\{1,\dots,n\}$.  By (S1) again,
this implies that $P_{k+1}(a)\ge P_k(a)+1$ and that
\[
 \uP(q)=\uP(a)+(q-a)\ue_k \quad (q\in[a,a+1]).
\]
Therefore, the half line $[0,\infty)$ can be partitioned
in maximal intervals $[q_i,q_{i+1})$ ($0\le i <s$) on
which (iv) holds for some $k_i\in\{1,\dots,n\}$.
The existence of an integer $\ell_{i+1}\in\{1,\dots,n\}$
satisfying (iii) then follows by the continuity of the
map $\uP$.  Finally, the condition in (ii) expresses
the maximality of those intervals thanks to (S3).
\end{proof}

We can now state and prove the following converse to
Theorem \ref{theorem:nsys}.

\begin{theorem}
\label{thm:construction}
Let $\uP$ be as in the previous lemma.  Then there exists
 a point $\uu\in\Kv^n$ of norm $1$ such that $\uP=\uL_\uu$.
\end{theorem}

\begin{proof}
Using the notation of the previous lemma, we first
construct recursively, for each integer $i$ with
$0\le i< s$, a basis $(\ux_1^{(i)},\dots,\ux^{(i)}_n)$ of $A^n$
with the following properties:
\begin{itemize}
 \item[(B1)]
  $(\ux_1^{(i)},\dots,\widehat{\ux^{(i)}_{k_i}},\dots,\ux^{(i)}_n,\ue_n)$
  is an orthogonal basis of $\Kv^n$,
 \item[(B2)]
  $\log \bnorm{\ux^{(i)}_j} = P_j(q_i)$ for $j=1,\dots,n$,
 \item[(B3)]
  $(\ux_1^{(i)},\dots,\widehat{\ux_{\ell_{i}}^{(i)}},\dots,\ux^{(i)}_n)
   = (\ux_1^{(i-1)},\dots,\widehat{\ux_{k_{i-1}}^{(i-1)}},\dots,\ux^{(i-1)}_n)$
   \ \text{if \ $i\ge 1$.}
\end{itemize}
For $i=0$, we choose $(\ux_1^{(0)},\dots,\ux^{(0)}_n)=(\ue_1,\dots,\ue_n)$.
Then the conditions are fulfilled because $k_0=n$, $q_0=0$ and $P_j(0)=0$
for $j=1,\dots,n$.  Suppose now that $i\ge 1$ and that appropriate bases
have been constructed for all smaller values of the index.  By
Lemma \ref{lemma:description_nsys}, we have
\begin{equation}
\label{thm:construction:eq1}
 \left(
   P_1(q_i),\dots,\widehat{P_{\ell_{i}}(q_i)},\dots,P_n(q_i)
 \right)
   =
 \left(
   P_1(q_{i-1}),\dots,\widehat{P_{k_{i-1}}(q_{i-1})},\dots,P_n(q_{i-1})
 \right)
\end{equation}
and $P_{\ell_i}(q_i) \ge P_{\ell_i}(q_{i-1})
   = \max\{P_1(q_{i-1}), \dots,P_{\ell_i}(q_{i-1})\}$
as well as $P_{\ell_i}(q_i) > P_{k_{i-1}}(q_{i-1})$.
In view of the induction hypothesis, this yields
\[
 P_{\ell_i}(q_i)
   \ge
   \max\big\{\log\bnorm{\ux_1^{(i-1)}},\dots,\log\bnorm{\ux_{\ell_i}^{(i-1)}}\big\}
 \et
 P_{\ell_i}(q_i)
   >
   \log\bnorm{\ux_{k_{i-1}}^{(i-1)}}
\]
Since $k_{i-1}\le\ell_i$ and $k_i<\ell_i$,
Lemma \ref{lemma:construction_bases} then produces a basis
$(\ux_1^{(i)},\dots,\ux^{(i)}_n)$ of $A^n$ satisfying
(B1), (B3) and
\[
 \log \bnorm{\ux_{\ell_i}^{(i)}}=P_{\ell_i}(q_i).
\]
Thus it also satisfies (B2) because
of \eqref{thm:construction:eq1}
combined with (B3) and the induction hypothesis
that $\log \norm{\ux_j^{(i-1)}}=P_j(q_{i-1})$ for $j=1,\dots,n$.

For each index $i$ with $0\le i<s$, let $\uu_i$ denote
an element of $\Kv^n$ of norm $1$ with
$\uu_i\cdot\ux_j^{(i)}=0$ for each $j=1,\dots,n$
with $j\neq k_i$.  By Lemma \ref{lemma:dist_dual} and (B3),
we have
\[
 \dist(\uu_i,\uu_{i-1})
 = \frac{\bnorm{\ux_1^{(i)}\wedge\cdots\wedge\ux_n^{(i)}}}
      {\bnorm{\ux_1^{(i)}}\cdots\bnorm{\ux_n^{(i)}}}
 \quad \text{if $i\ge 1$.}
\]
Since $(\ux_1^{(i)},\dots,\ux^{(i)}_n)$ is a basis of
$A^n$, its determinant belongs to $A^\times\subset\cOv^\times$
and so we obtain that
$\bnorm{\ux_1^{(i)}\wedge\cdots\wedge\ux_n^{(i)}}=1$.
Then, using (B2), we conclude that
\begin{equation}
 \label{thm:construction:eq2}
 \dist(\uu_i,\uu_{i-1})
 = \exp(-P_1(q_i)-\cdots-P_n(q_i))=\exp(-q_i)
 \quad (1\le i<s).
\end{equation}
Since $k_0=n$ and $(\ux_1^{(0)},\dots,\ux^{(0)}_n)
=(\ue_1,\dots,\ue_n)$, we may assume that
$\uu_0=\ue_n$.  Since $\dist(\uu_i,\uu_{i-1})<1$
when $1\le i<s$, Lemma \ref{lemma:comparison_distance_norm}
implies that $\abs{\uu_i\cdot\ue_n}=1$
for each of those $i$.  So, upon replacing $\uu_i$ by
$(\uu_i\cdot\ue_n)^{-1}\uu_i$, we may assume that
$\uu_i\cdot\ue_n=1$.  The norm of $\uu_i$ remains equal to $1$,
and the same lemma combined with \eqref{thm:construction:eq2} gives
\[
 \norm{\uu_i-\uu_{i-1}}
   = \dist(\uu_i,\uu_{i-1})
   = \exp(-q_i)
   \quad
   (1\le i<s).
\]
Moreover, $(q_i)_{0\le i<s}$ is a strictly increasing sequence
of non-negative integers. So, if $s=\infty$, the sequence
$(\uu_i)_{i\ge 0}$ converges in norm to an element $\uu$
of $\Kv^n$ of norm $1$ with
\[
 \norm{\uu_i-\uu} = \exp(-q_{i+1})
 \quad
 (0\le i<s).
\]
If $s<\infty$, the latter inequalities remain true for the choice
of $\uu=\uu_{s-1}$ upon setting $q_s=\infty$.  We claim that
the vector $\uu$ has the requested property.

To show this, let $q\ge 0$ be an arbitrary non-negative integer,
and let $i$ be the index with $0\le i<s$ such that $q_i\le q<q_{i+1}$
(with the above convention that $q_s=\infty$ if $i=s-1<\infty$).
For each $j\in\{1,\dots,n\}$ with $j\neq k_i$, we have
$\uu_i\cdot\ux_j^{(i)}=0$, thus
\[
 \babs{\uu\cdot\ux_j^{(i)}}
  = \babs{(\uu-\uu_i)\cdot\ux_j^{(i)}}
  \le \norm{\uu-\uu_i}\,\bnorm{\ux_j^{(i)}}
   = \exp(-q_{i+1}) \bnorm{\ux_j^{(i)}}
   < {\mathrm{e}}^{-q} \bnorm{\ux_j^{(i)}},
\]
and so
\[
 L_{\ux_j^{(i)}}(q)
  = \max\left\{ \log\bnorm{\ux_j^{(i)}},\
       q+\log \babs{\uu\cdot\ux_j^{(i)}}\right\}
  = \log\bnorm{\ux_j^{(i)}}
  =P_j(q_i).
\]
If $i\ge 1$, we also have $\uu_{i-1}\cdot\ux_{k_i}^{(i)}=0$
because of (B3), and a similar computation gives
\[
 \babs{\uu\cdot\ux_{k_i}^{(i)}}
   \le {\mathrm{e}}^{-q_i} \bnorm{\ux_{k_i}^{(i)}}.
\]
This inequality still holds if $i=0$ because, in that case, its
right hand side is $1$.  So, in all cases we find that
\begin{equation}
 \label{thm:construction:eq3}
 L_{\ux_{k_i}^{(i)}}(q)
  \le q-q_i+\log \bnorm{\ux_{k_i}^{(i)}}
  = q-q_i+P_{k_i}(q_i).
\end{equation}
Since $(\ux_1^{(i)},\dots,\ux_n^{(i)})$ is a basis of $A^n$,
this implies that, for the componentwise partial ordering
on $\bR^n$, we have
\begin{align*}
 \uL_{\uu}(q)
  &\le \Phi_n\left(
       L_{\ux_1^{(i)}}(q),\dots,L_{\ux_n^{(i)}}(q)
       \right) \\
  &\le \Phi_n\left(
       P_1(q_i),\dots,\widehat{P_{k_i}(q_i)},\dots,P_n(q_i),
       q-q_i+P_{k_i}(q_i)
       \right) \\
  &= \uP(q).
\end{align*}
Since the components
of $\uL_\uu(q)$ and of $\uP(q)$ both add up to $q$, this implies that
$\uL_\uu(q)=\uP(q)$ as announced.  Moreover, we must have equality
in \eqref{thm:construction:eq3}.
\end{proof}

Like the proof of lemma \ref{lemma:construction_bases},
the above argument is entirely deterministic in the sense
that it yields a single point $\uu$ with the requested
properties.  Moreover, if $F_0$ denotes the smallest subfield
of $F$, then each $n$-tuple $(\ux_1^{(i)},\dots,\ux^{(i)}_n)$
that it constructs is in fact a basis of $F_0[T]^n$ over
$F_0[T]$, and the corresponding approximation $\uu_i$ of $\uu$ with
$\uu_i\cdot\ue_n=1$ belongs to $F_0(T)^n$.  So these can be
calculated recursively on a computer for a given $n$-system $\uP$.
We further develop this remark below.

\subsection{Universality of the construction}

Let $\uP=(P_1,\dots,P_n)\colon [0,\infty)\to \bR^n$ be
an $n$-system such that $\uP(q)\in\bZ^n$ for each
integer $q\ge 0$.  We claim that, when $F=\bQ$, the point
$\uu$ of $\bQ((1/T))^n$ provided by the proof of
Theorem \ref{thm:construction} belongs in fact to
$\bZ[[1/T]]^n$ and that, for a general field $F$, the
point that it produces is its image $\bar{\uu}\in\Kv^n$ under
the reduction of coefficients from $\bZ$ to $F$.

By induction on $i$, we first note that, when $F=\bQ$,
the $n$-tuples $(\ux_1^{(i)},\dots,\ux^{(i)}_n)$
attached to $\uP$
are bases of $\bZ[T]^n$ and that, for a general field
$F$, the corresponding $n$-tuples are their images
$(\bar{\ux}_1^{(i)},\dots,\bar{\ux}^{(i)}_n)$
under the reduction of coefficients from $\bZ$ to $F$.
When $F=\bQ$, the point $\uu_i$ is the last row in
the inverse transpose of the matrix $M_i$ whose rows are
$\ux_1^{(i)},\dots,\widehat{\ux_{k_{i}}^{(i)}},\dots,
\ux^{(i)}_n,\ue_n$.  However, the condition 5) in
Lemma \ref{lemma:construction_bases} implies that
$\det(M_i)$ is a monic polynomial of $\bZ[T]$ for each
index $i$ with $0\le i<s$.  Thus each $\uu_i$ has coefficients
in $\bZ[[1/T]]$ and the same is true of the vector $\uu$.
In particular, it makes sense to consider their images
$\bar{\uu}_i$ and $\bar{\uu}$ under reduction. Clearly
we have $\bar{\uu}_i\cdot\ue_n=1$ and
$\bar{\uu}_i\cdot\bar{\ux}_j^{(i)}=0$ for each
$j=1,\dots,n$ with $j\neq k_i$.  Thus, we have
$\uL_\uu=\uP$ when working in $\bQ((1/T))^n$ and
$\uL_{\bar{\uu}}=\uP$ when working in $F((1/T))$.

\begin{remark}
Although our construction yields a single point $\uu$ with
$\uL_\uu=\uP$, such a point $\uu$ is far from being unique.
Consider for example an arbitrary $2$-system $\uP=(P_1,P_2)\colon
[0,\infty)\to \bR^2$ for which $P_1$ is unbounded.
There is a unique sequence of integers
$d_0=0<d_1<d_2<\cdots$ such that, upon putting $q_0=0$ and
$q_i=d_{i-1}+d_i$ for each $i\ge 1$, we have
\[
 \uP(q)=\Phi_2(d_i,q-d_i)
 \quad\text{for any}\quad
 i\ge 0
 \et
 q\in[q_i,q_{i+1}].
\]
With this notation, one can check that the point $\uu$
constructed in the proof of Theorem \ref{thm:construction} is
$\uu=(-\xi_0,1)$ where $\xi_0\in\cO_\infty$ has the continued
fraction expansion
\[
 \xi_0=[0,T^{d_1-d_0},T^{d_2-d_1},\dots].
\]
However, the continued fraction $\xi=[a_0,a_1,a_2,\dots]$
has the same property for any sequence
$(a_i)_{i\ge 0}$ in $A=F[T]$ satisfying $a_0\in F$ and
$\deg(a_i)=d_i-d_{i-1}$ for each $i\ge 1$.  Clearly
the point $\uu=(-\xi,1)$ then has $\|\uu\|=1$.  To show that
$\uL_\uu=\uP$, define recursively $\uy_{-1}=(0,1)$,
$\uy_0=(1,a_0)$ and $\uy_i=a_i\uy_{i-1}+\uy_{i-2}$ for each
$i\ge 1$. Then the theory of continued fractions shows that,
with respect to $\uu$, one has
\[
 L_{\uy_{-1}}(q)=q \et L_{\uy_i}(q)=\max\{d_i,q-d_{i+1}\}
 \quad (q\ge 0,\ i\ge 0).
\]
So, for a given integer $i\ge 0$ and a given
$q\in [q_i,q_{i+1}]$, we have $L_{\uy_{i-1}}(q)=q-d_i$
and $L_{\uy_i}(q)=d_i$.  Since $\uy_{i-1}$ and $\uy_i$
form a basis of $A^2$, this implies that, for the
componentwise ordering on $\bR^2$, we have
$\uL_\uu(q)\le \Phi_2(d_i,q-d_i)=\uP(q)$,
and so $\uL_\uu(q)=\uP(q)$ (because both points have the
sum of their coordinates equal to $q$).
\end{remark}

\section{Duality and an alternative normalization}
\label{sec:duality}

Let $\uu\in\Kv^n$ with $\norm{\uu}=1$. It can be
shown that, for each $q\in\bN=\{0,1,2,\dots\}$, the dual of
the convex body $\cC_\uu({\mathrm{e}}^q)$ defined in \S\ref{ssec:Lu}
is
\[
 \cC^*_\uu({\mathrm{e}}^q)
 = \{\uy\in\Kv^n\,;\,
     \norm{\uy}\le {\mathrm{e}}^q
     \ \text{and}\
     \norm{\uu\wedge\uy}\le 1\}.
\]
For each $j=1,\dots,n$ and each $q\in[0,\infty)$, we define
$L^*_{\uu,j}(q)$ to be the minimum of all $t\in \bR$
for which the inequalities
\[
 \norm{\uy}\le {\mathrm{e}}^{q+t}
 \et
 \norm{\uu\wedge\uy}\le {\mathrm{e}}^{t}
\]
admit at least $j$ linearly independent solutions
$\uy$ in $A^n$ so that, when $q\in\bN$, this is
the logarithm of the $j$-th minimum of $\cC^*_{\uu}({\mathrm{e}}^q)$.
Then Theorem \ref{thm:Mahler:duality} gives
\[
 (L^*_{\uu,1}(q),\dots,L^*_{\uu,n}(q))
 = (-L_{\uu,n}(q),\dots,-L_{\uu,1}(q))
\]
for each $q\in \bN$.
This remains true for all $q\in [0,\infty)$ because
a reasoning similar to that in \S\ref{ssec:Lu} shows
that, like $\uL_\uu$, the map $\uL^*_\uu
= (L^*_{\uu,1},\dots,L^*_{\uu,n})$ is affine in each
interval between two consecutive integers.

The analogue of the setting of Schmidt and Summerer
in \cite{SS2013} would require instead to work with the
family of convex bodies of volume $1$ given by
\[
 T^{-q}\cC^*_\uu({\mathrm{e}}^{nq})
 = \{\uy\in\Kv^n\,;\,
     \norm{\uy}\le {\mathrm{e}}^{(n-1)q}
     \ \text{and}\
     \norm{\uu\wedge\uy}\le {\mathrm{e}}^{-q}\}
 \quad
 (q\in\bN).
\]
Associate to this family is the map $\tuL_\uu
=(\tL_{\uu,1},\dots,\tL_{\uu,n})\colon [0,\infty)\to\bR^n$
where $\tL_{\uu,j}(q)$ is the minimum of all $t\in \bR$
for which the inequalities
\[
 \norm{\uy}\le {\mathrm{e}}^{(n-1)q+t}
 \et
 \norm{\uu\wedge\uy}\le {\mathrm{e}}^{-q+t}
\]
admit at least $j$ linearly independent solutions
$\uy$ in $A^n$, and thus $\tL_{\uu,j}(q)=q+L^*_{\uu,j}(nq)$.

 \section{Perfect systems}
\label{sec:perfect}

From now on, we work with several places of $K=F(T)$.  So, we
distinguish the corresponding absolute values with subscripts.
For each $\alpha\in F$, we denote by $K_\alpha=F((T-\alpha))$
the completion of $K$ for the absolute value
$|f|_\alpha={\mathrm{e}}^{-\ord_\alpha(f)}$ where, for $f$ in $K$ or in
$K_\alpha$, the quantity $\ord_\alpha(f)\in\bZ\cup\{\infty\}$
represents the order of $f$ at $\alpha$ (with the convention
that $\ord_\alpha(0)=\infty$).
We also write $|\ |_\infty$ for the absolute value on $K$
and on $\Kv=F((1/T))$ previously denoted without subscript,
so that $|f|_\infty={\mathrm{e}}^{\deg(f)}$ for any series $f\in\Kv$.
For each $\alpha\in F\cup\{\infty\}$ and each integer
$n\ge 1$, we equip $K_\alpha^n$ with the maximum norm
denoted $\|\ \|_\alpha$.

Let $\uf=(f_1,\dots,f_n)$ be an $n$-tuple of elements of
$F[[T]]$.  A linear algebra argument shows that,
for any non-zero $(\varrho_1,\dots,\varrho_n)\in\bN^n$,
there exists a non-zero point $\ua=(a_1,\dots,a_n)$ in
$A^n=F[T]^n$ such that
\begin{equation}
 \label{prefect:eq:alglin}
 \deg (a_i)\le \varrho_i-1 \quad (1\le i\le n)
 \et
 \ord_0(\ua\cdot\uf) \ge \varrho_1+\cdots+\varrho_n-1.
\end{equation}
Following Mahler \cite{Ma1968} and Jager \cite{Ja1964}, we say
that $\uf$ is \emph{normal} for $(\varrho_1,\dots,\varrho_n)$
if any non-zero solution $\ua$ of \eqref{prefect:eq:alglin}
in $A^n$ has $\ord_0(\ua\cdot\uf)= \varrho_1+\cdots+\varrho_n-1$.
Then, those solutions together with $0$ constitute, over $F$,
a one dimensional subspace of $A^n$.  We also say that $\uf$ is a
\emph{perfect system} if it is normal for any
$(\varrho_1,\dots,\varrho_n)\in\bN^n\setminus\{0\}$.

\begin{examples} Suppose that $F$ has characteristic zero.
If $\omega_1,\dots,\omega_n$ are elements of $F$ then
\begin{equation*}
 \label{perfect:eq:exp}
 ({\mathrm{e}}^{\omega_1 T},\dots, {\mathrm{e}}^{\omega_n T})
 \quad\text{where}\quad
 {\mathrm{e}}^{\omega T}=\sum_{j\ge 0} \frac{\omega ^j}{j!}T^j
\end{equation*}
is a perfect system \cite[Theorem 1.2.1]{Ja1964}.
If moreover $\omega_1,\dots,\omega_n$ are pairwise
incongruent modulo $\bZ$ then
\begin{equation*}
 \label{perfect:eq:bin}
 ((1+T)^{\omega_1},\dots, (1+T)^{\omega_n})
 \quad\text{where}\quad
 (1+T)^\omega= \sum_{j=0}^\infty \binom{\omega}{j} T^j,
\end{equation*}
is also a perfect system \cite[Theorem 1.2.2]{Ja1964}.
Finally the $n$-tuple
\begin{equation*}
 \label{perfect:eq:log}
 \left((\log(1-T))^{n-1},\dots, \log(1-T),1\right)
 \quad\text{where}\quad
 \log(1-T)= -\sum_{j=1}^\infty \frac{T^j}{j}
\end{equation*}
is normal for each $(\varrho_1,\dots,\varrho_n)
\in\bN^n\setminus\{0\}$ with $\varrho_1\le\cdots\le\varrho_n$
\cite[Theorem 1.2.3]{Ja1964}.
When $F=\bC$, the first example of a perfect system is due to Hermite in \cite{He1893}, although it
also follows by duality from his earlier work on the transcendence of ${\mathrm{e}}$ in \cite{He1873} (see also \cite{Ma1931}).
To our knowledge, no perfect $n$-system of series
of $F[[T]]$ with $n\ge 2$ is known when $F$ is a finite field.
A short computation shows that there are none when $F$ has two
or three elements.
\end{examples}

In view of the first example above, Theorem B in the
introduction follows from the following result which
also applies to the two other examples as well as to
any perfect system.

\begin{theorem}
Let $\uf=(f_1(T),\dots,f_n(T))\in F[[T]]^n$ with $n\ge 2$.
Suppose that $\uf$ is normal for each diagonal element
$(\varrho,\dots,\varrho)\in\bN^n\setminus\{0\}$.
Then the point $\uu=(f_1(1/T),\dots,f_n(1/T))\in\Kv^n$
satisfies $\|\uu\|_\infty=1$ and its associated map $\uL_\uu$
is the unique $n$-system $\uP$ characterized
by the property \eqref{intro:eq:extremalP}.
\end{theorem}

\begin{proof}
Since $\uf$ is normal for $(1,\dots,1)$, we have
$\|\uf\|_0=1$, thus $\|\uu\|_\infty=\|\uf\|_0=1$.
Fix $q\in\bN$ and let $t=L_{\uu,1}(q)\in\bN$.  By definition there
exists a non-zero point $\ux=(x_1(T),\dots,x_n(T))$ in $A^n$ such
that
\[
 \|\ux\|_\infty\le {\mathrm{e}}^t
 \et
 |\ux\cdot\uu|_\infty \le {\mathrm{e}}^{t-q}.
\]
Then, for each $i=1,\dots,n$, the polynomial
$a_i(T)=T^tx_i(1/T)$ satisfies $\deg(a_i(T))\le t$ and
we find that
\[
 \ord_0\big(a_1(T)f_1(T)+\cdots+a_n(T)f_n(T)\big)
   = t - \deg\big(\ux\cdot\uu\big)
   \ge q.
\]
Since $\uf$ is normal for $(t+1,\dots,t+1)$, this implies that
$n(t+1)>q$ or equivalently that
\[
 L_{\uu,1}(q)
  \ge \left\lfloor \frac{q}{n} \right\rfloor
 \quad
 (q\in\bN).
\]
For $q=mn$ with $m\in\bN$, this gives $L_{\uu,1}(mn)\ge m$
and, since the coordinates of $L_\uu(mn)$ form a monotone
increasing sequence with sum $mn$, all of these
are equal to $m$, in particular $L_{\uu,1}(mn)=L_{\uu,n}(mn)=m$.
Now let $q\ge 0$ be any real number and let $m\in\bN$ such
that $mn\le q\le (m+1)n$.  Since $L_{\uu,1}$ and $L_{\uu,n}$
are monotone increasing, we find
\[
 L_{\uu,n}(q)-L_{\uu,1}(q)
 \le L_{\uu,n}((m+1)n)-L_{\uu,1}(mn) = 1.
\]
As observed in the introduction, this characterizes $\uL_\uu$
as the $n$-system described in there.
\end{proof}

In the case where $\uf$ is normal for each
$(\varrho_1,\dots,\varrho_n)\in\bN^n\setminus\{0\}$
with $\varrho_1\le\cdots\le\varrho_n$ and $\varrho_n\le \varrho_1+1$,
it is also possible to relate the points which
realize the successive minima to the corresponding solutions
of \eqref{prefect:eq:alglin}.  To this end, we note that each integer
$i\ge 1$ can be written as a sum $i=\varrho_{i,1}+\cdots+\varrho_{i,n}$
for a unique such $n$-tuple given by
$\varrho_{i,j}=\lceil(i+j-n)/n\rceil$ for $j=1,\dots,n$.   Define
$\uy_i=T^{\varrho_{i,n}-1}(a_{i,1}(1/T),\dots,a_{i,n}(1/T))$ where
$\ua_i=(a_{i,1},\dots,a_{i,n})$ is a corresponding non-zero solution
of \eqref{prefect:eq:alglin}.  Then $\uy_i\in A^n$ because
$\deg(a_{i,j})\le \varrho_{i,n}-1$
for $j=1,\dots,n$.  Moreover, we have
\[
 \|\uy_i\|_\infty
  ={\mathrm{e}}^{\varrho_{i,n}-1}\|\ua_i\|_0
  ={\mathrm{e}}^{\lceil i/n\rceil-1}
 \et
 |\uy_i\cdot\uu|_\infty
  ={\mathrm{e}}^{\varrho_{i,n}-1}|\ua_i\cdot\uf|_0
  ={\mathrm{e}}^{\lceil i/n\rceil-i}
\]
because $\|\ua_i\|_0=1$ and $|\ua_i\cdot\uf|_0={\mathrm{e}}^{-i+1}$.
Thus, with respect to the point $\uu$, we deduce that
\[
 L_{\uy_i}(q)=\max\{\lceil i/n\rceil-1,q+\lceil i/n\rceil-i\}
 \quad
 (q\ge 0, i\ge 1).
\]
In particular the trajectory of $\uy_i$ changes slope
from $0$ to $1$ at the point $q=i-1$.
The hypothesis also implies that $\deg(a_{i,j})\le
\lceil(i+j-2n)/n\rceil$ for each $i\ge 1$ and
each $j=1,\dots,n$, with equality when
$i+j\equiv 1 \mod n$. This in turn implies that
$\det(\ua_i,\dots,\ua_{i+n-1})$ is a non-zero
polynomial of degree $i-1$ for each $i\ge 1$.
Thus, the points $\uy_i,\uy_{i+1},\dots,\uy_{i+n-1}$
are linearly independent over $K$ and so, for
each $q\in [i-1,i]$, we obtain
\[
\begin{aligned}
 \uL_\uu(q)
 &\le \Phi_n\big(L_{\uy_i}(q),\dots,L_{\uy_{i+n-1}}(q)\big)\\
 &= \Phi_n\Big(q+\Big\lceil \frac{i}{n}\Big\rceil-i,
      \Big\lceil \frac{i+1}{n}\Big\rceil-1, \dots,
      \Big\lceil \frac{i+n-1}{n}\Big\rceil-1 \Big).
\end{aligned}
\]
Since the arguments of $\Phi_n$ in the last expression add
up to $q$, we conclude that the latter is equal to $\uL_\uu(q)$.
Therefore
$\uy_i,\uy_{i+1},\dots,\uy_{i+n-1}$ realize the minima
of $\cC_\uu({\mathrm{e}}^q)$ for $q=i-1$ and for $q=i$, while their
trajectories cover the combined graph of $\uL_\uu$ over
the interval $[i-1,i]$.

 \section{An adelic estimate}
\label{sec:adelic}

In this section we assume that $F=\bC$ so that, for each
$\omega$ and $\alpha$ in $\bC$, we may define
\[
 {\mathrm{e}}^{\omega T}
  := {\mathrm{e}}^{\omega\alpha}\sum_{j=0}^\infty \frac{\omega^j}{j!}(T-\alpha)^j
  \in \bC[[T-\alpha]].
\]
We also fix an integer $n\ge 1$ and $n$ distinct complex
numbers $\omega_1,\dots,\omega_n\in\bC$.  Our last
main result is the following.

\begin{theorem}\label{Theorem:adelic}
Let $S=\{\alpha_1,\dots,\alpha_s\}$ be a finite
subset of $\bC$ of cardinality $s\ge 1$.
Then, for any $n$-tuple of non-zero polynomials
$\ua=(a_1(T),\dots,a_n(T))$ in $\bC[T]$, we have
\[
 |a_1|_\infty\cdots|a_n|_\infty
 \prod_{j=1}^s
  \Big(
  \|\ua\|_{\alpha_j}^{-1}
  |a_1|_{\alpha_j}\cdots|a_n|_{\alpha_j}
  |\ua\cdot\uf|_{\alpha_j}
  \Big)
 \ge C(n)^{-s}
\]
where $\uf=({\mathrm{e}}^{\omega_1T},\dots,{\mathrm{e}}^{\omega_nT})$
and $C(n)=\exp(n(n-1)/2)$.
\end{theorem}

\begin{proof}
Fix a choice of non-zero polynomials $a_1,\dots,a_n$ in
$\bC[T]$.  Put $\ua=(a_1,\dots,a_n)$ and,
for $i=1,\dots,n$, let $c_iT^{d_i}$ denote
the leading monomial of $a_i(T)$.  For each $k\in\bN$,
we write
\[
 \Big(\frac{d}{dT}\Big)^k
   \big(a_1(T){\mathrm{e}}^{\omega_1T}+\cdots+a_n(T){\mathrm{e}}^{\omega_nT}\big)
 = a_{k,1}(T){\mathrm{e}}^{\omega_1T}+\cdots+a_{k,n}(T){\mathrm{e}}^{\omega_nT}
\]
where $a_{k,i}(T)=(\omega_i+d/dT)^ka_i(T)
= \omega_i^kc_iT^{d_i}\, +\,
(\text{terms of lower degree})$.  Define
\[
 \ua_k=(a_{k,1}(T),\dots,a_{k,n}(T))
 \quad
 (0\le k<n),
\]
and put $\Delta=\det(\ua_0,\dots,\ua_{n-1})$. Then $\Delta$
is a non-zero polynomial of degree $d=d_1+\dots+d_n$ whose
coefficient of $T^d$ is the product of $c_1\cdots c_n\neq 0$
with the Vandermonde determinant
$\det(\omega_i^k)\neq 0$ (using the convention that $0^0=1$
if $\omega_i=0$ for some $i$). Thus we have
\[
 |\Delta|_\infty = |a_1|_\infty\cdots|a_n|_\infty.
\]

Now fix a choice of $j\in\{1,\dots,s\}$.  Put
$\alpha=\alpha_j$ and choose $\ell\in\{1,\dots,n\}$
such that $\|\ua\|_{\alpha}=|a_\ell|_\alpha$.  Define also
\[
 \ub_k=(\ua_k\cdot\uf,a_{k,1},\dots,\widehat{a_{k,\ell}},\dots,a_{k,n})
 \quad
 (0\le k<n).
\]
Since $|{\mathrm{e}}^{\omega_\ell T}|_\alpha=1$, we have
$|\Delta|_\alpha=|\det(\ub_0,\dots,\ub_{n-1})|_\alpha$.
On the other hand, since $\ua_k\cdot\uf$ is the $k$-th
derivative of $\ua\cdot\uf$, we have
\[
 \ord_\alpha(\ua_k\cdot\uf)
   \ge \ord_\alpha(\ua\cdot\uf) - k
 \quad
 (0\le k<n),
\]
and similarly
\[
 \ord_\alpha(a_{k,i})\ge \ord_\alpha(a_i)-k
 \quad
 (0\le k<n,\ 1\le i\le n).
\]
From this we deduce that
\[
 \ord_\alpha(\Delta)
  \ge
  -\binom{n}{2} + \ord_\alpha(\ua\cdot\uf)
  + \ord_\alpha(a_1)+\cdots+\widehat{\ord_\alpha(a_\ell)}
  +\cdots+\ord_\alpha(a_n),
\]
and thus
\[
 |\Delta|_\alpha
  \le
  C(n) \|\ua\|_{\alpha}^{-1}
  |a_1|_{\alpha}\cdots|a_n|_{\alpha}
  |\ua\cdot\uf|_{\alpha}
 \quad
 (\alpha\in\{\alpha_1,\dots,\alpha_s\}).
\]
The conclusion follows because the product
formula yields $1\le |\Delta|_\infty
|\Delta|_{\alpha_1}\cdots|\Delta|_{\alpha_s}$.
\end{proof}

\begin{remark}
Under the assumptions of Theorem \ref{Theorem:adelic}, the above
argument also yields
\[
 |a_1|_\infty\cdots|a_n|_\infty
 \prod_{j=1}^s
  |\ua\cdot\uf|_{\alpha_j}
 \ge C'(n)^{-s}
\]
with $C'(n)=\exp(n-1)$.  The latter estimate is best possible
for any choice of $n,s\ge 1$ as one sees by expanding $({\mathrm{e}}^T-1)^{n-1}$
in the form $\ua\cdot\uf$ with $\omega_j=j-1$ and
$a_j(T)=\binom{n-1}{j-1}(-1)^{n-j}$ for $j=1,\dots,n$ and
by choosing the points $\alpha_j=2\pi ji$ for $j=1,\dots,s$.
Then we have $|a_j|_\infty=1$ for $j=1,\dots,n$ and
$|\ua\cdot\uf|_{\alpha_j}=C'(n)^{-1}$ for $j=1,\dots,s$.
This construction shows that the constant $C(n)$ in
Theorem \ref{Theorem:adelic} cannot be replaced by a number
less than $\exp(n-1)$.
\end{remark}

By a change of variables, we deduce from
Theorem \ref{Theorem:adelic} the following
statement involving the functions ${\mathrm{e}}^{\omega_i/T}$.

\begin{corollary}
Let $\ua=(a_1(T),\dots,a_n(T))$ be an $n$-tuple of
non-zero polynomials in $\bC[T]$.  Then, we have
\[
 |a_1|_0\cdots|a_n|_0
 |a_1|_\infty\cdots|a_n|_\infty
 |\ua\cdot\uu|_\infty
 \ge C(n)^{-1} \|\ua\|_\infty,
\]
where $\uu=({\mathrm{e}}^{\omega_1/T},\dots,{\mathrm{e}}^{\omega_n/T})$
and where $C(n)$ is as in the theorem.
\end{corollary}

\begin{proof}
Let $d$ be the largest of the degrees of $a_1,\dots,a_n$.
Set
\[
 \ux=(x_1,\dots,x_n)=(T^da_1(1/T),\dots,T^da_n(1/T))
 \et
 \uf=({\mathrm{e}}^{\omega_1T},\dots,{\mathrm{e}}^{\omega_nT}).
\]
Since $x_1,\dots,x_n$ are non-zero polynomials, the
preceding theorem gives
\[
 |x_1|_\infty\cdots|x_n|_\infty
 |x_1|_0\cdots|x_n|_0
 |\ux\cdot\uf|_0
 \ge C(n)^{-1} \|\ux\|_0.
\]
The conclusion follows because, for each $i=1,\dots,n$,
we have $\deg(x_i)=d-\ord_0(a_i)$
and $\ord_0(x_i)=d-\deg(a_i)$, thus $|x_i|_\infty|x_i|_0=
|a_i|_0|a_i|_\infty$, while $\|\ux\|_0={\mathrm{e}}^{-d}\|\ua\|_\infty$
and $|\ux\cdot\uf|_0= {\mathrm{e}}^{-d} |\ua\cdot\uu|_\infty$.
\end{proof}

We conclude with two sets of inequalities, the second
one being the result announced by Baker in \cite{Ba1964}
and proved there in the case $n=3$, except for the value
of the constant.

\begin{corollary}
Let $a_1(T),\dots,a_n(T)$ be
non-zero polynomials in $\bC[T]$.  Then, we have
\begin{align*}
 &\big|a_1(T){\mathrm{e}}^{\omega_1/T}+\cdots+a_n(T){\mathrm{e}}^{\omega_n/T}\big|_\infty
  \prod_{i=2}^n |a_i(T)|_\infty
  \ge C(n)^{-1}, \\
 &|a_1(T)|_\infty
  \prod_{i=2}^n
  \big|a_1(T){\mathrm{e}}^{\omega_i/T}-a_i(T){\mathrm{e}}^{\omega_1/T}\big|_\infty
  \ge C(n)^{-(n-1)}.
\end{align*}
\end{corollary}

\begin{proof}
The first estimate follows directly from the previous
corollary using the facts that $|a_i|_0\le 1$ for each
$i=1,\dots,n$ and that $\|\ua\|_\infty\ge |a_1|_\infty$.
It implies that, within $\Kv=\bC((1/T))$, the series
$u_1={\mathrm{e}}^{\omega_1/T}, \dots, u_n={\mathrm{e}}^{\omega_n/T}$
are linearly independent over $\bC(T)$. Consequently,
for each $(g_1,\dots,g_n)\in\bZ^n$, the sets
\[
 \begin{aligned}
 \cC
 &= \{ (x_1,\dots,x_n)\in\Kv^n\,;\,
     |x_1u_1+\cdots+x_nu_n|_\infty\le {\mathrm{e}}^{g_1}
     \text{ and }
     |x_i|_\infty\le {\mathrm{e}}^{g_i} \ (2\le i\le n)\}, \\
 \cC^*
 &= \{ (y_1,\dots,y_n)\in\Kv^n\,;\,
     |y_1|_\infty\le {\mathrm{e}}^{-g_1}
     \text{ and }
     |y_1u_i-y_iu_1|_\infty\le {\mathrm{e}}^{-g_i} \ (2\le i\le n)\}
 \end{aligned}
\]
are dual convex bodies of $\Kv^n$.  Moreover, the same estimate
implies that the first minimum $\lambda_1$ of $\cC$ satisfies
$\lambda_1^nV \ge C(n)^{-1}$ where $V={\mathrm{e}}^{g_1+\cdots+g_n}$
is the volume of $\cC$.  By Theorems \ref{thm:Mahler:Minkowski}
and  \ref{thm:Mahler:duality},
this implies that the first minimum $\lambda^*_1$ of $\cC^*$
satisfies
\[
 \lambda^*_1
   = \lambda_n^{-1}
   = \lambda_1\cdots\lambda_{n-1}V
   \ge \lambda_1^{n-1}V
   \ge C(n)^{-(n-1)/n} V^{1/n}.
\]
Upon choosing $g_1,\dots,g_n$ so that $|a_1|_\infty={\mathrm{e}}^{-g_1}$
and $|a_1u_i-a_iu_1|_\infty={\mathrm{e}}^{-g_i}$ for $i=2,\dots,n$, we
also have $\lambda^*_1\le 1$, and so we obtain $V\le C(n)^{n-1}$
which yields the second inequality of the corollary.
\end{proof}

\vfill
 \vfill

\hbox{
\small
\vbox{
\hbox{Damien \sc Roy}
	\hbox{D\'epartement de math\'ematiques et de statistique \qquad \qquad \qquad
}
	\hbox{Universit\'e d'Ottawa
}
	\hbox{585, Avenue King Edward
}
	\hbox{Ottawa, Ontario
}
	\hbox{Canada K1N 6N5
	}
}	
\hfill \qquad
\vbox{	\hbox{Michel \sc Waldschmidt
	}
	\hbox{Sorbonne Universit\'es
	}
	\hbox{UPMC Univ Paris 06
	}
	\hbox{UMR 7586 IMJ-PRG
	}
	\hbox{F -- 75005 Paris, \sc France
	}
	\hbox{}
}	
}	

\end{document}